\documentclass[a4paper,oneside]{amsart}

\usepackage{amssymb}
\usepackage{amsmath}
\usepackage{amsfonts}
\usepackage{amsthm}
\usepackage{mathrsfs}
\usepackage{fix-cm}
\usepackage{graphicx}
\usepackage{amscd}
\usepackage{enumerate}

\usepackage{soul}
\sodef\spred{}{.2em}{.9em plus.4em}{1em plus.1em minus.1em}

\usepackage[latin1]{inputenc}
\usepackage[T1]{fontenc}

\usepackage[english]{babel}

\usepackage{bbm}

\usepackage[all]{xy}

\newbox\mybox
\def\overtag#1#2#3{\setbox\mybox\hbox{$#1$}\hbox to
  0pt{\vbox to 0pt{\vglue-#3\vglue-\ht\mybox\hbox to \wd\mybox
      {\hss$\ss#2$\hss}\vss}\hss}\box\mybox}
\def\undertag#1#2#3{\setbox\mybox\hbox{$#1$}\hbox to 0pt{\vbox to
    0pt{\vglue#3\vglue\ht\mybox\hbox to \wd\mybox
      {\hss$\ss#2$\hss}\vss}\hss}\box\mybox}
\def\lefttag#1#2#3{\hbox to 0pt{\vbox to 0pt{\vss\hbox to
      0pt{\hss$\ss#2$\hskip#3}\vss}}#1}
\def\righttag#1#2#3{\hbox to 0pt{\vbox to 0pt{\vss\hbox to
      0pt{\hskip#3$\ss#2$\hss}\vss}}#1}
\let\ss\scriptstyle

\def\Dot{\lower.2pt \hbox to 3.5pt{\hss$\bullet$\hss}}
\def\Circ{\lower.2pt \hbox to 3.5pt{\hss$\circ$\hss}}

\def\splicediag#1#2{\xymatrix@R=#1pt@C=#2pt@M=0pt@W=0pt@H=0pt}

\newcommand\lineto{\ar@{-}}
\newcommand\dashto{\ar@{--}}
\newcommand\dotto{\ar@{.}}

\newtheorem{thm}{Theorem}[section]

\newtheorem{prop}[thm]{Proposition}
\newtheorem{cor}[thm]{Corollary}

\newtheorem{thm*}{Theorem}

\theoremstyle{definition} 
\newtheorem{defn}[thm]{Definition} 
\newtheorem{ex}[thm]{Example}

\newcommand{\C}{\mathbbm{C}}

\newcommand{\morf}[4][\to]{ #2 \colon #3 #1 #4}
\newcommand{\inv}{^{-1}}

\newcommand{\M}[1]{M^{ #1 }_{m,n}}
\newcommand{\G}{\operatorname{Gr}}
\newcommand{\Tj}{\operatorname{Tjur}}
\newcommand{\Coker}{\operatorname{Coker}}
\newcommand{\im}{\operatorname{Im}}
\newcommand{\Na}{\operatorname{Nash}}
\newcommand{\Sp}[1]{\operatorname{Span}\{ #1\}}
\newcommand{\rank}{\operatorname{rank}}
\newcommand{\id}{\operatorname{id}}
\newcommand{\codim}{\operatorname{codim}}
\newcommand{\Gl}{\operatorname{GL}}

\renewcommand{\phi}{\varphi}
\renewcommand{\epsilon}{\varepsilon}

\begin{document}

\bibliographystyle{alpha}

\title{On Tjurina Transform and Resolution of Determinantal Singularities}
\author{Helge M\o{}ller Pedersen}
\address{ICMC-USP}
\email{helge@imf.au.dk}
\begin{abstract} 
Determinantal singularities are an important class of singularities,
generalizing complete intersections, which recently have seen a large
amount of interest. They are defined as preimage of $\M{t}$ the sets of
matrices of rank less than $t$. The linear algebraic
structure $\M{t}$ gives rise to some interesting structures on
determinantal singularities. In this
article we will focus on one of these, namely the \emph{Tjurina
  transform}. We will show some properties of it, and discuss how it
can and how can not be used to find resolutions of determinantal
singularities.
\end{abstract}
\maketitle

\section{Introduction}

Hypersurface singularities have in general been the starting point of
singularity
theory. They have some very good properties, one of the most important
is the existence of the Milnor fibration \cite{milnor}. The Milnor
fibration makes it possible to define the Milnor number $\mu$, which is
a very important invariant. So a goal in singularity theory is to find
more general families of singularities, for which it is possible to
define the Milnor number. A classical example of a generalization, for
which the Milnor number can be defined, is the complete
intersections. Determinantal singularities are a generalization of the
complete intersections, they are defined as the preimage of the
set of $m\times n$ matrices of rank less than $t$ under certain
holomorphic maps. They have seen a lot of interest lately, there have
been several different ways to define the Milnor number of certain
classes of determinantal varieties by Ruas and da Silva Pereira
\cite{cedinhamiriam}, Damon and Pike \cite{damonpike} and
Nu{\~n}o-Ballesteros, Or{\'e}fice-Okamoto and Tomazella
\cite{NunoOreficeOkamotoTomazella}, Ebeling and
Gusein-Zade defined the index of a 1-form \cite{ebelingguseinzade},
and in addition to these the deformation theory has been studied in
\cite{gaffenyrangachev}.

In this article we study other aspects of determinantal
singularities, not directly related to deformation theory, namely,
transformations and resolutions. They played a very important role in
\cite{ebelingguseinzade}, and the Tjurina transform, which will be one
of our main subjects, was also studied intensely for the case
Cohen-Macaulay codimension 2 by Frühbis-Krüger and Zach in
\cite{fruhbiskrugerzach}. 

We first study the Tjurina, Tjurina transpose and Nash
transformations for the model determinantal singularity in section
\ref{modelsing}. This was already done in \cite{ebelingguseinzade},
but we will need this as motivation for introducing the
transformations for general determinantal singularities. We also
explore how these transformations are related and how they are not,
and give a description of their homotopy type. We introduce the
Tjurina transform (and its transpose) for general determinantal
singularities in section \ref{generalsing}, give some general
properties, for example that the Tjurina transform of a complete
intersection is itself, and give some methods to
find the Tjurina transform. In section \ref{completeintersection} we
show that under some general assumptions the Tjurina transform or its dual
is a complete intersection. This unfortunately mean that Tjurina transform
is not going to provide resolutions in general. At last we illustrate
in section \ref{hypersurface} that by changing the
determinantal type of the Tjurina transform of certain hypersurface
singularities, we can continue the process of taking Tjurina
transform, and in the end reach a resolution. Section
\ref{preliminaries} introduces the determinantal singularities and
notions of transformations used throughout the article. I wish to
thank Maria Ruas for introducing me to the subject of determinantal
singularities and for many fruitful conversations during the preparation
of this article. The author was supported by FAPESP grant 2015/08026-4.  

\section{Preliminaries}\label{preliminaries}

In this section we give the basic definitions and properties of
determinantal singularities, and transformations we will need. We
will in general follow the notation Ebeling and Gusein-Zade used in
\cite{ebelingguseinzade}. 

\subsection{Determinantal Singularities}

Let $M_{m,n}$ be the set of $m\times n$ matrices over $\C$, then we
define the \emph{model determinantal singularity of type} $(m,n,t)$,
denoted by $\M{t}$, to be the subset of $M_{m,n}$ consisting of matrices $A$ of
$\rank(A)< t$. $\M{t}$ has a natural structure of an irreducible algebraic
variety, with defining equations given by requiring that the
$t\times t$ minors have to vanish. The dimension of $\M{t}$ is
$mn-(m-t+1)(n-t+1)$. The model determinantal
singularities are often called generic determinantal singularities as
for example in \cite{cedinhamiriam}. 

The singular set of $\M{t}$ is $\M{t-1}$ and the decomposition of
$\M{t}=\bigcup_{i=1}^{t}(\M{i}-\M{i-1})$, where $\M{0}:=\emptyset$, is a
Whitney stratification.

Let $\morf{F}{U\subset\C^N}{M_{m,n}}$ be a map with holomorphic entries. We
say that $X:=F\inv(\M{t})$ is a \emph{determinantal singularity of
  type} $(m,n,t)$ if $\codim(X)=\codim(\M{t})=(m-t+1)(n-t+1)$. $X$ has
the structure of an irreducible algebraic variety, with equations
defined by the vanishing of the $t\times t$ minors of the matrix
$F(x)$. The singular set of $X$ is $F\inv(\M{t-1})$, the decomposition
$X=\bigcup_{i=1}^t X^i$, where $X^i=F\inv(\M{i}-\M{i-1})$, is a
  stratification. The group $\Gl_m(\mathcal{O}_N) \times \Gl_n(\mathcal{O}_N)$,
  acting by conjugation, acts on the set of determinantal varieties of type
  $(m,n,t)$ by isomorphism.   

If $F$ intersects the strata $\M{i}-\M{i-1}$ transversally at $F(x)$, then
the singularity a $x$ only depends upon $\rank(F(x))$. We, therefore,
call such a point \emph{essentially nonsingular}. This naturally
leads to the next definition.
\begin{defn}
Let $X$ be a determinantal singularity defined by the map $F$. Then
$X$ is an \emph{essentially isolated determinantal singularity} (or EIDS
for short) if all points $x\in X-\{ 0\}$ are essentially nonsingular.
\end{defn}

An EIDS is of course not smooth, but the singularities away from $\{0 \}$
are controlled, i.e.\ they only depend on the strata they belong to. An
example of an EIDS is any complete intersection given the type of a
$(1,m,1)$ (or $(m,1,1)$) determinantal singularity.

If $(X,0)$ is a determinantal singularity of type $(m,n,t)$ given by
$\morf{F}{U\subset \C^N}{M_{m,n}}$ satisfying $F(0)\neq 0$, then one can find
another map $\morf{F'}{U'\subset\C^N}{M_{m-s,n-s}}$ with $F'(0)=0$
such that $F'$ gives
$(X,0)$ the structure of a determinantal singularity of type
$(m-s,n-s,t-s)$ where $U$ and $U'$ are open
neighbourhoods of the origin and $s=\rank F(0)$. This can be done by conjugating
$F$ to be on the form $\left(
\begin{array}{@{} c | c @{}}
\id_s & 0 \\ \hline
0 & F'
\end{array}
\right)$ in a neighbourhood of $0$.

\subsection{Transformations}

As the article is about Tjurina transforms and resolutions of
determinantal singularities, we will define what we mean by a transformation.
\begin{defn}\label{transform}
Let $X$ be a variety and $A\subset X$ a closed subvariety, then a
\emph{transformation of} $(X,A)$ is a variety $\widetilde{X}$ together with a
proper map $\morf{\pi}{\widetilde{X}}{X}$, such that
$\morf{\pi}{\pi\inv(X-A)}{X-A}$ is an isomorphism and
$\overline{\pi\inv(X-A)}=\widetilde{X}$.
\end{defn}
The last requirement insures that $\dim(\pi\inv(A))<\dim(X)$.

A resolution of $(X,Sing X)$ is then just a transformation where
$\widetilde{X}$ is smooth. We want to compare the different
transformations, so we define a map between transformations as follows.
\begin{defn}\label{mapoftransformations}
Let $\morf{f}{T_1}{T_2}$ be a map between two different transformations
$\morf{\pi_i}{(T_i,E_i)}{(X,A)}$ of the same space and
subspace. Then we call $f$ a map of transformations if
$\pi_1=\pi_2\circ f$. We call a map of transformation an isomorphism,
if it is an isomorphism of varieties.
\end{defn}

\section{Resolutions of the Model Determinantal Singularities}\label{modelsing}

In \cite{ebelingguseinzade} Ebeling and Gusein-Zade introduce 3
different natural ways to resolve 
the model determinantal singularities $\M{t}$. The first is the same
as the Tjurina transform of $(\M{t},\M{t-1})$ used by Tjurina
\cite{tjurina}, Van Straten \cite{vanstraten} and 
Frühbis-Krüger and Zach in \cite{fruhbiskrugerzach}. It is
defined as the following set in $M_{m,n}\times \G (n-t+1,n)$: 
\begin{align*}
\Tj(\M{t}): &= \{ (A,W)\in M_{m,n}\times \G (n-t+1,n)\ \vert\
A(W)=0\}\\
&= \{ (A,W)\in M_{m,n}\times \G (n-t+1,n)\ \vert\ W \subset
\ker(A)\}
\end{align*}
by considering $A\in \M{t}$ as a linear map $\morf{A}{\C^n}{\C^m}$. It
is shown in \cite{arbarellocornalbagriffiths}, that this is a smooth
variety. Consider $\morf{\pi}{\Tj(\M{t})}{\M{t}}$ the restriction of
the projection to the first factor. Then over the regular part of
$\M{t}$ we have that the map $A\to (A,\ker A)$ is an inverse to $\pi$,
hence $\morf{\pi}{\Tj(\M{t})}{\M{t}}$ is a resolution. Corollary 3.3
in \cite{fruhbiskrugerzach} shows that their definition gives the same as
this one, their proof also works for general $n,m$.

The second resolution is as the Tjurina, but considering $A\in \M{t}$
as a linear map $\morf{A}{\C^m}{\C^n}$. This is of course the map given
by the transpose of $A$, so we get the following:
\begin{align*}
\Tj^T(\M{t}): &= \{ (A,W)\in M_{m,n}\times \G (m-t+1,m)\ \vert\
A^T(W)=0\}\\
&= \{ (A,W)\in M_{m,n}\times \G (m-t+1,m)\ \vert\ W \subset
\ker(A^T)\}\\
&= \{ (A,W)\in M_{m,n}\times \G (m-t+1,m)\ \vert\ W \subset
\Coker(A)\}
\end{align*}  
It is clear from the definition that this is also a smooth variety,
the same proof as in the case of Tjurina transform works. If one
chooses an inner product on $\C^m$, then one gets that $ W \subset
\Coker(A)$ is the same as $\im(A)\subset W^\bot $ where $W^\bot$ is
the orthogonal complement with respect to the inner product. The choice
of inner product also gives an isomorphism between $\G(m-t+1,m)$ and
$\G(t-1,m)$ defined by sending $W$ to $W^\bot$. Using this we get that
this transform is also:
\begin{align}
\Tj^T(\M{t}) &= \{ (A,W)\in M_{m,n}\times \G (t-1,m)\ \vert\
\im(A)\subset W\}.\label{tjurim}
\end{align}   

The third resolution considered by Ebeling and Gusein-Zade is the Nash
transform of $\M{t}$. In section one of \cite{ebelingguseinzade} they
show how to get the Nash transform which can be stated as the
following proposition:
\begin{prop}\label{nashcar}
For a model determinantal singularity the Nash transform can be given
as the following:
\begin{align*}
\Na(\M{t}) = \{  (A,W_1,W_2)\in M_{m,n}\times &\G(n-t+1,n)\times
\G(t-1,m)\\
 &\vert\ \ker(A)\supset W_1 \text{ and } \im(A) \subset W_2\}.
\end{align*}
\end{prop}

\begin{proof}
In \cite{arbarellocornalbagriffiths} they show that for
$A\in\M{t}-\M{t-1}$, that is the
regular points, that $T_A\M{t}=\{ B\in M_{m,n}\ \vert\ B\big( \ker
(A)\big)\subset \im(A) \}$. Consider the map
$\morf{\alpha}{\G (n-t+1,n)\times \G (t-1,m)}{\G (d^t_{m,n},mn)}$, where
$d^t_{m,n} : = mn-(m-t+1)(n-t+1)=\dim( \M{t})$, given by
$\alpha(W_1,W_2)=\{ B\in M_{m,n}\ \vert\ B(W_1)\subset
W_2\}$. 

We will first show that $\alpha$ is injective. Assume that
there exist two pairs $(W_1,W_2)$ and $(V_1,V_2)$ such that
$\alpha(W_1,W_2)=\alpha(V_1,V_2)$. Assume that $W_1\neq V_1$, let
$v_1\in V_1$ and $v_1\notin W_1$, since $\dim(W_1)=\dim(V_1)$ such $v_1$
exists, and choose $v_2\notin V_2$. Define the linear map $B$ as the
map that sends $av_1$ to $av_2$ and anything else to $0$. Then
$B(W_1)=\{0\}\subset W_2$ and hence $B\in \alpha(W_1,W_2)$, but
$B(V_1)=\Sp{v_2}\not\subset V_2$, so $B\notin \alpha(V_1,V_2)$ so we
have a contradiction. Assume now that there exist pairs $(W_1,W_2)$ and
$(W_1,V_2)$ such that $\alpha(W_1,W_2)=\alpha(W_1,V_2)$. Assume that
$W_2\neq V_2$, let $v_1\in W_1$ and let $v_2\in V_2$ and $v_2\notin
W_2$, since $\dim(W_2)=\dim(V_2)$ such $v_2$ exists. Define $B$ as the
linear map that sends $av_1$ to $av_2$. Then $B(W_2)=\Sp{v_2}\subset
V_2$ so $B\in \alpha(W_1,V_2)$, but $\Sp{v_2}\not\subset W_2$ so
$B\notin \alpha(W_1,W_2)$ so we have a contradiction. This shows that
$\alpha$ is injective.

Next we will show that $\alpha$ is continuous. Let
$(V_i,W_i)\in\G(n-t+1,n)\times\G(t-1,m)$ be a convergent sequence and
let $(V,W)=\lim (V_i,W_i)$. Let $\mathcal{B}_i=\alpha(V_i,W_i)$, and
choose a convergent subsequence $\mathcal{B}'_i$ which exists because
$\G(d^t_{m,n},mn)$ is compact. Let $\mathcal{B}=\lim
\mathcal{B}'_i$, choose $B\in\mathcal{B}$ and $B_i\in\mathcal{B}'_i$ a
sequence of matrices converging to $B$. Let $v\in V$ and $v_i\in V_i$
a sequence converging to $v$, set $w_j=B_jv_j$ for any $j$ where $B_j$
is defined. Now since $B_j$ and $v_j$ converge, $w_j$ converges to
$w=Bv$, but $w_j\in W_j$ and hence its limit is in $W$. So for all
$v\in V$ and all $B\in\mathcal{B}$ $Bv\in W$, hence
$\mathcal{B}\subset \alpha(V,W)$, but since $\dim
\big(\mathcal{B} \big) =\dim\big( \alpha(V,W)\big)$ we have that
$\mathcal{B}=\alpha(V,W)$. So any convergent subsequence of
$\mathcal{B}_i$ converges to $\alpha(V,W)$, this implies that
$\mathcal{B}_i$ converges to $\alpha(V,W)$ since $\G(d^t_{m,n},mn)$ is
compact. Therefore, $\lim \alpha(V_i,W_i) =\alpha(\lim
(V_i,W_i))$ for all convergent sequences, hence $\alpha$ is continuous.

Since $\alpha$ is a continuous map from a compact space to a compact
space it is closed, and since it is injective it implies it is an
embedding.

Let $\morf{\beta}{(\M{t}- \M{t-1})}{M_{m,n}\times \G(n-t+1,n) \times
  \G(t-1,m)}$ be the map $\beta(A)=\big(A,\ker(A),\im(A)\big)$. Let
$\morf{\alpha'}{M_{m,n}\times\G(n-t+1,n)\times\G(t-1,m)}{
  M_{m,n}\times\G(d^t_{m,n},mn)}$ defined by
  $\alpha'(A,V,W)=(A,\alpha(V,W))$. Then 
$\alpha'\circ\beta(A)=(A, \mathcal{B})$, where
\begin{align*}
\mathcal{B} &=\alpha\big(\ker(A),\im(A)\big)=\{ B\in M_{m,n}\ \vert \
B\big(\ker(A)\big)\subset \im(A)\} = T_A \M{t}
\end{align*}
So $\alpha'\circ\beta$ is the same as the Gauss map on the regular
part of $\M{t}$. Then we have that
$\Na(\M{t})=\overline{ \alpha'\circ\beta (\M{t}-\M{t-1})}$. Since
$\alpha$ and hence $\alpha'$ is a closed embedding we have
$\Na(\M{t})=\alpha'\Big(\overline{\beta (\M{t}-\M{t-1})}\Big)$. Moreover,
since $\alpha'$ is an embedding it follows that $\Na(\M{t})$ is
homeomorphic to $\overline{\beta (\M{t}-\M{t-1})}$.

The last part of the proof is determining $\overline{\beta (\M{t}-\M{t-1})}$.
Now $\beta (\M{t}-\M{t-1})=\{(A,\ker A,\im A)\in M_{m,n}\times \G(n-t+1,n)\times
\G(t-1,m)\}$ and we want to show that the closure is  $\{  (A,V,W)\in M_{m,n}\times \G(n-t+1,n)\times
\G(t-1,m)\
 \vert\ \ker(A)\supset V \text{ and } \im(A) \subset W\}=\mathcal{N}$. First
 assume that $(A,V,W)\in \overline{\beta (\M{t}-\M{t-1})}$ is not in
 $\mathcal{N}$. This implies that that there is a $v\in V$ such that
 $Av\notin W$. Let $(A_i,V_i,W_i)$ be a sequence in $(A,V,W)\in
 \overline{\beta (\M{t}-\M{t-1})}$ converging to $(A,V,W)$ and $v_i\in
 V_i$ a sequence converging to $v$, then $A_iv_i$ converge to $Av$
 but $A_iv_i=0$ so this contradicts $Av\notin\mathcal{N}$. Let
 $(A,V,W)\in\mathcal{N}$ and let $r=\rank A$. Now $V\subset \ker A$,
 so let $V'\subset \C^n$ be the subspace satisfying $V\oplus V'= \ker
 A$, and $\im A\subset W$ so let $W'\subset \C^m$ be the subspace
 satisfying $\im A\oplus W'=W$. Let $A'$ be a matrix of rank $t-1-r$,
 such that $\ker A'\oplus V' =\C^n$ and $\im A' = W'$, such a matrix
 exist since $\dim V'=\dim W'=t-1-r$. Set $A_i=A+\tfrac{1}{i}A'$ then
 $\ker A_i= \ker A\bigcap \ker \tfrac{1}{i}A'=\ker A\bigcap\ker A'=V$
 and $\im A_i=\im A+\im \tfrac{1}{i}A'=W$. Hence $(A_i,V_i,W_i)
 := (A_i,V,W)$ is a sequence in $\beta (\M{t}-\M{t-1})$
 converging to $(A,V,W)$, so $\mathcal{N}\subset \overline{\beta
   (\M{t}-\M{t-1})}$ which finishes the proof. 
\end{proof}

An important consequence of this is the following:
\begin{cor}
$\Na(\M{t})$ is smooth.
\end{cor}

\begin{proof}
Using the description of $\Na(\M{t})$ given in Proposition
\ref{nashcar} we get that the projection to the two last factors
$\G(n-t+1,n)\times \G(t-1,m)$ gives $\Na(\M{t})$ the structure of the
total space of a vector bundle over a smooth manifold.
\end{proof}

It follows from Definition \ref{mapoftransformations} and Proposition
\ref{nashcar}, that we have a map of transformations
$\morf{f}{\Na(\M{t})}{\Tj\M{t}}$ by setting $f(A,V,W)=(A,V)$ and a map
of transformations $\morf{f}{\Na(\M{t})}{\Tj^T\M{t}}$ by setting
$f(A,V,W)=(A,W)$ and using \eqref{tjurim}. These maps are never
isomorphism, as we will see later when we determine the homotopy type
of these spaces. Now finding maps between $\Tj\M{t}$ and $\Tj^T\M{t}$
turns out to be impossible by the following result:
\begin{prop}
There exist no continuous map of transformations between $\Tj\M{t}$ and
$\Tj^T\M{t}$.
\end{prop}

\begin{proof}
We start by using \eqref{tjurim} to identify $\Tj^T(\M{t})$ with $\{
(A,W)\in M_{m,n}\times \G (t-1,m)\ \vert\ \im(A)\subset W\}$. Let
$\morf{f}{\Tj(\M{t})}{\Tj^T\M{t}}$ be a map of transformations, this
implies that over $\pi\inv(\M{t}-\M{t-1})$ we have that $f(A,\ker
A)=(A,\im A)$. Let $\{x_1,\dots,x_n\}$ be a basis of $\C^n$ and
$\{y_1,\dots, y_m\}$ be a basis for $\C^m$. Let $A$ be the matrix in
this basis of the linear map$A(x_1,\dots,x_n)=(x_1,\dots,x_{t-2},0,
\dots,0)$, notice that there is at least $2$ zeros at the end since
$t\leq m$. $\rank A=t-2$ hence $A\in\M{t-1}$. Let
$V=\Sp{x_t,\dots,x_{n}}$ then it is clear that $\ker A\supset V$.

We now define two different sequences of matrices $A_s^1$ and
$A_s^2$. $A_s^1(x_1,\dots,x_n) :=
(x_1,\dots,x_{t-2},\tfrac{1}{s}x_{t-1},0,0,\dots,0)$ and
$A_s^2(x_1,\dots,x_n) := (x_1,\dots,x_{t-2},0,\tfrac{1}{s}x_{t-1},0,\dots,0)$. It
is clear that $\ker A_s^i = V$ and $\lim_{s\to\infty} (A_s^i,V) = (A,V)$
for $i=1,2$. Since $A_s^i\in \M{t}-\M{t-1}$ we get that
$f(A_s^i,V) = (A_s^i,\im A_s^i)$. Now let
$W_1 := \Sp{y_1,\dots,y_{t-1}} = \im A_s^1$ and
$W_2 := \Sp{y_1,\dots,y_{t-2}, y_t} = \im A_s^2$. If $f$ is continuous,
then we have that $f(A,W)=f(\lim_{s\to\infty}
(A_s^i,V))=\lim_{s\to\infty} f(A_s^i,V) = (A,W_i)$ for $i=1,2$. But
$W_1\neq W_2$ hence $f$ can not be continuous. The argument that
there is no continuous map of transformations from $\Tj^T\M{t}$ to
$\Tj\M{t}$ is similar.  
\end{proof}

Next we determine the homotopy type of the transformations, and
the above shows that even in the case $\Tj\M{t}$ and $\Tj^T\M{t}$ are
homotopy equivalent they are not isomorphic.

\begin{prop}
Let $\morf{\pi}{(T(\M{t}),E)}{(\M{t},\M{t-1})}$ be one of the $3$
transforms discussed above. Then $T(\M{t})$ is homotopy equivalent to
$\pi\inv(0)$. 
\end{prop}
This gives that $\Na(\M{t})\sim \G( n-t+1,n)\times \G(t-1,m)$,
$\Tj(\M{t})\sim \G(n-t+1,n)$ and $\Tj^T(\M{t})\sim \G(t+1,m)$, where
$\sim$ denotes homotopy equivalence.

\begin{proof}
We will only show this for $\Na(\M{t})$ but the other proofs are
similar. Let $\morf{F}{\Na(\M{t})\times\C}{\Na(\M{t})}$ be the map
defined by $F(A,V,W,s)=f_s(A,V,W)=(sA,V,W)$, when we use the
identification for the Nash transformation given by Proposition
\ref{nashcar}. The map is well defined since $(sA)(V)=s(A(V))=0$ and $\im
sA=\im A\subset W$ if $s\neq 0$ and $\im sA=\{0\}\subset W$ if $s=0$. It is
continuous since it is just scalar multiplication. Restrict the map to
$s\in[0,1]$. Then $f_1=\id$, $f_s\vert_{\pi\inv(0)} =id\vert_{\pi\inv(0)}$ and
$f_0(\Na(\M{t}))=\pi\inv(0)$. Hence $f_s$ is a deformation retraction,
and $\Na(\M{t})$ is homotopy equivalent to $\pi\inv(0)$.  
\end{proof}

\section{Transformations of general determinantal singularities}\label{generalsing}

In this section we will introduce the transformations defined above
for general determinantal varieties. We start by introducing the
Tjurina transform. The Tjurina transform of determinantal
singularities has be introduced several places before for example
\cite{tjurina}, \cite{vanstraten},
\cite{arbarellocornalbagriffiths},
\cite{ebelingguseinzade} and \cite{fruhbiskrugerzach}. They in general define
the Tjurina transform of a determinantal singularity $X$ of type
$(m,n,t)$ given by $\morf{F}{\C^n}{M_{m,n}}$ as the fibre product
$X\times_F\Tj(\M{t})$, which works very well in the cases they
consider. But this definition gives the following
problem in a more general setting: Assume that $\dim(X)\leq
(t-1)(n-t+1)$ and let $\morf{p}{X\times_F\Tj(\M{t})}{X}$ be the
projection to the first factor. Then $p\inv(0)\cong \G(n-t+1,n)$,
hence the exceptional fibre of $p$ has dimension greater than or equal
to the dimension of $X$. This means that the fibre product does not
satisfy the conditions to be a transformation given in Definition
\ref{transform}. We will instead give an alternative definition
that does not have this problem. It should be said that in
\cite{ebelingguseinzade} and \cite{fruhbiskrugerzach} they only
consider the Tjurina transformation in situations where this does not
happen, and that our definition agrees with theirs in these cases. We
will in Proposition \ref{tjtilde=tj} see when the two definitions
agree in general. 

\begin{defn}
Let $X$ be a determinantal singularity of type $(m,n,t)$ given by
$\morf{F}{\C^n}{M_{m,n}}$, define $\morf{B}{X_{reg}}{\G(t-1,n)}$ as
the map that sends $x$ into the row space of $F(x)$. Then we define
the Tjurina transform $\Tj(X)$ of $X$ as
\begin{align*}
\Tj(X) &:= \overline{\bigg{\{} \big(x,W\big)\in X_{reg}\times\G(t-1,n)\ \vert\
    W=B(x)\bigg{\}} }\subset X\times \G(t-1,n),
\end{align*}
and we define the map $\morf{\pi^{Tj}}{\Tj(X)}{X}$ as the projection
to the first factor.
\end{defn}
It is clear that this satisfies the assumptions of Definition
\ref{transform} to be a transformation of $(X,X_{sing})$ since
$\pi^{Tj}\vert_{\Tj(X)-\pi\inv_{Tj}(X_{sing})}$ is the inverse of $B$,
it is proper since all fibres are either points or closed subsets of
$\G(t-1,n)$ hence compact, and
$\dim\Big((\pi^{Tj})\inv(X_{sing})\Big)<\dim(X)$ since it is a closed subvariety
of an algebraic variety of $\dim(X)$. 

Notice that the choice of an inner product on $\C^n$ gives a one to
one correspondence between the row space of $F(x)$ and $\ker F(x)$ and
an isomorphism $\G(t-1,n)\cong\G(n-t+1,n)$. Hence we get that
\begin{align*}
\Tj(X) &= \overline{\bigg{\{} \big(x,W\big)\in X_{reg}\times\G(n-t+1,n)\ \vert\
    W=\ker F(x)\bigg{\}} }.%\subset X\times \G(n-t+1,n).
\end{align*}
We use the row space in our definition, since it makes calculation
easier as we see later.

\begin{prop}\label{completeintersection}
Let $X$ be a complete intersection and not a hypersurface of type
$(m,m,m)$, then $\Tj(X)= X$.
\end{prop}

\begin{proof}
A complete intersection is a determinantal singularity of type
$(m,n,1)$ if it is not a hypersurface of type $(m,m,m)$. Hence
$\Tj(X)\subset X\times\G(0,n) =X$ and $B$ is constant. The result then
follows since $\Tj(X)=\overline{X_{reg}}$ and the regular points
is an open dense subset of $X$. 
\end{proof}
We will later see some examples of hypersurfaces of type $(m,m,m)$
where the Tjurina transform is useful in simplifying singularities.

To study the local properties of the Tjurina transform closer we will
use the following matrix charts on $\G(t-1,n)$. Let $I\subset
\{1,\dots,n\}$ such that $\#I=t-1$. For each such $I=\{
i_1,\dots,i_{t-1}\}$ we define a chart of $\G(t-1,n)$ by the
$(t-1)\times n$ matrix $A_I$ which consists of the columns $C_i$
given as follows: 
\begin{align*}
C_i=
\begin{pmatrix}
a_{1i}\\
\vdots\\
a_{(t-1)i}\\
\end{pmatrix}
\text{ if } i\notin I, \text{ and } C_{i_l}=\begin{pmatrix}
0\\
\vdots\\
0\\
1\\
0\\
\vdots\\
0\\
\end{pmatrix}
\text{the } l^{th} \text{ entry.}
\end{align*}
Here we consider $a_{ji}\in \C$ as variables. Let
$a=(a_{1i_1},\dots,a_{(t-1)i_{n-t+1}})\in \C^{(t-1)(n-t+1)}$, hence
$A_I(a)$ define a map $\morf{\tilde{A}_I}{\C^{(t-1)(n-t+1)}}{\G(t-1,n)}$
by sending $a$ to the row space of
$A_I(a)$. $\big{\{}\tilde{A}_I\big{\}}_I$ is a cover of $\G(t-1,n)$ by algebraic
maps, and if $U_I=\im \tilde{A}_I$ the change of coordinates from
  $\tilde{A}_I\inv (U_I\bigcap U_J)$ to $\tilde{A}_J\inv (U_I\bigcap
  U_J)$ is given by $A_J^TA_I(a)$. 

To see the row space of $F(x)$ in a given chart $A_I$, we construct
the following $(m+t-1)\times n$ matrix:
\begin{align*}
\tilde{F}_I^{Tj}(x,a)=
\left(
\begin{array}{@{} c  @{}}
 A_I(a)\\ \hline
F(x)
\end{array}
\right).
\end{align*}
Then the row space of $F(x)$ is contained in $\tilde{A}_I(a)$ if and
only if $\rank \tilde{F}_I^{Tj}(x,a)=t-1$. 

Let $\widetilde{\Tj}_I(X)=
(\tilde{F}_I^{Tj})\inv\big(M^{t}_{m+t-1,n}\big)\subset
X\times\C^{(t-1)(n-t+1)}$, and
$\morf{\tilde{\pi}^{Tj}_I}{\widetilde{\Tj}_I(X)}{X}$ be the projection
to the first factor. 

It is clear from the above construction that
$\Tj_I(X):=\Tj(X)\bigcap \im \tilde{A}_I\subset \widetilde{\Tj}_I(X)$
but they are not necessarily equal, moreover, $\widetilde{\Tj}_I(X)$ is not
necessarily a determinantal singularity. We have that
$(\pi^{Tj}_I)\inv(X_{reg})=(\tilde{\pi}^{Tj}_I)\inv(X_{reg})$ this implies that
$\dim  \widetilde{\Tj}_I(X)=\max\big(\dim X,\dim
(\tilde{F}_I^{Tj})\inv(M^{t-1}_{m+t-1,n})\big)$. Now $\dim
(\tilde{F}_I^{Tj})\inv(M^{t-1}_{m+t-1,n})$ is the largest of the
dimensions of the pullback of the strata of $M^{t-1}_{m+t-1,n}$. So
$(\tilde{F}_I^{Tj})\inv(M^{s}_{m+t-1,n}-M^{s-1}_{m+t-1,n})\subset
X\times\G(t-1,n)$ consist of the pairs $(x,W)$ such that $x\in
F\inv(M^s_{m,n}-M^{s-1}_{m+t-1,n})$ and $\{\text{row space of }F(x)\}\subset
  W$. Since $\rank F(x)=s-1$ we can write all such $W$ as
  $W=\{\text{row space of   }F(x)\}+W_{F(x)}$ where $W_{F(x)}$ is a $t-1-s$
  dimensional subspace of the complement of $\{\text{row space of
  }F(x)\}\subset\C^n$. Moreover, we have that for any $V$ in the complement of
  $\{\text{row space of }F(x)\}\subset\C^n$ we have that $\rank \{\text{row
    space of   }F(x)\}+V=t-1 $. Hence we get that
  $\big{\{}W\in\G(t-1,n)\ \vert\ \{\text{row space of }F(x)\}\subset
  W\big{\}}$ is isomorphic to $\G(t-1-s,n-s)$. So we get that $\dim
  (\tilde{F}_I^{Tj})\inv(M^{s}_{m+t-1,n}-M^{s-1}_{m+t-1,n})\big) = \dim
  F\inv((M^{s}_{m+t-1,n}-M^{s-1}_{m+t-1,n}) + \dim \G(t-1-s,n-s)$.

The above implies that $\dim  \widetilde{\Tj}_I(X)=\dim \Tj_I(x) =
\dim X$ if and only if $\dim F\inv((M^{s}_{m+t-1,n}-M^{s-1}_{m+t-1,n})
\leq \dim X - \dim \G(t-1-s,n-s) =N -(m-t+1)(n-t+1)-(t-s)(n-t+1) =
N-(m-s+1)(n-t+1)$ for all $s=1,\dots,t$. If $X$ has an isolated
singularity, this becomes $N\geq m(n-t+1)$.

\begin{prop}
If $\dim  \widetilde{\Tj}_I(X)=\dim X$ then
$\widetilde{\Tj}_I(X)$ is a determinantal singularity.
\end{prop}

\begin{proof}
We just need to check if $\codim \widetilde{\Tj}_I(X)=\codim
M^{t}_{m+t-1,n} = (m+t-1-t+1)(n-t+1)=m(n-t+1)$. But $\codim
\widetilde{\Tj}_I(X) =\codim \Tj_I(X)=\codim X+ (t-1)(n-t+1)=
(m+t-1)(n+t-1)+ (t-1)(n+t-1)=m(n-t+1)$.
\end{proof}

In this case we get that $\widetilde{\Tj}_I(X)$ is a determinantal
singularity of type $(m+t-1,n,t)$. But $\rank \tilde{F}_I^{Tj}(0,0) =
t-1$, so one can find another matrix $F'_I(x,a)$ defining
$\widetilde{\Tj}_I(X)$ such $F'_I(0,0)= 0$ and this is a determinantal
singularity of type $(m+t-1-(t-1),n+(t-1),t-(t-1))=(m,n+t-1,1)$. Since
$\codim \widetilde{\Tj}_I(X) =m(n-t+1)$ we get that
$\widetilde{\Tj}_I(X)$ is a complete intersection. We will later show
how to explicitly find $F'_I(x,a)$ also in the case $\dim
\widetilde{\Tj}_I(X) \neq\dim X$.

We can also use this to determine when $\Tj_I(X)$ and
$\widetilde{\Tj}_I(X)$ are equal. Notice that
$\widetilde{\Tj}_I(X)=\Big(X\times_F\Tj(\M{t})\Big)\bigcap \im
\tilde{A}_I$, hence the next proposition also answers the question,
when is our definition of Tjurina transform the same as the one used
by other authors.
\begin{prop}\label{tjtilde=tj}
$\widetilde{\Tj}_I(X)=\Tj_I(X)$ if and only if $\dim X^s<
N-(m-s+1)(n-t+1)$ for all $s\in 1,\dots,t-1$.
\end{prop}

\begin{proof}
Since $\Tj(X)$ is a transformation, we have that $\dim
\pi^{Tj}(X^{t-1})<\dim X$, remember the $X^{t-1}$ is the singular set
of $X$. Then the above calculations of the dimensions of the fibres
give the inequalities, and we get the only if direction.

So assume that the inequalities are satisfied, this implies that $\dim
\Tj_I(X)=\dim \widetilde{\Tj}_I(X)$ and that $\dim
(\widetilde{\pi}^{Tj}_I)\inv(X^{t-1})<\dim X$. Now $\Tj_I(X)$ is an
irreducible component of $\widetilde{\Tj}_I(X)$, and $\Tj_I(X)$ is not
a proper subvariety of any irreducible variety of the same dimension,
since it is closed. This implies that if $\widetilde{\Tj}_I(X)\neq\Tj_I(X)$ then
there exist an other irreducible component $V\subseteq
\widetilde{\Tj}_I(X)$ different from $\Tj_I(X)$. But since
$\widetilde{\Tj}_I(X)$ is a complete intersection it is
equidimensional, and hence $\dim V =\dim \Tj_I(X)$. Since
$(\pi^{Tj}_I)\inv(X_{reg})=(\tilde{\pi}^{Tj}_I)\inv(X_{reg})$ we have
that $V\subset (\tilde{\pi}^{Tj}_I)\inv(X^{t-1})$,  but this is a
contradiction since $\dim V> \dim \tilde{\pi}^{Tj}_I(X^{t-1})$. 
\end{proof}

We now want to give an explicit method to find $F_I'(x,a)$.
Assume that $I=i_1,\dots,i_{t-1}\subset \{1,\dots,n\}$ as before. Now
by adding columns of the form $-a_{ij} C_{i_j}$ to the $i$'th column,
for all $i\notin I$ and all $j= 1,\dots,t-1$, we get a matrix which has $t-1$
linearly independent rows $R_{i_j}$ of the form $R_{i_j}=(0,\dots,0,1,0,\dots,0)$,
where the $1$ is the $i_j$ entry. To this matrix we then add rows of
the form $-f_{ii_j}(x)R_{i_j}$ to the $i$'th row for $i=t,\dots,m$ and
$j=1,\dots,t-1$. We now have a matrix $\bar{F}_I(x,a)$ consisting of
the following columns:
\begin{align*}
\bar{F}_i=
\begin{pmatrix}
0\\
\vdots\\
0\\
f_{1i}(x)-\sum_{j=1}^{t-1}a_{ji}f_{1i_j}(x)\\
\vdots\\
f_{mi}(x)-\sum_{j=1}^{t-1}a_{ji}f_{mi_j}(x)\\
\end{pmatrix}
\text{ if } i\notin I, \text{ and } \bar{F}_{i_l}=\begin{pmatrix}
0\\
\vdots\\
0\\
1\\
0\\
\vdots\\
0\\
\end{pmatrix}
\text{the } l^{th} \text{ entry.}
\end{align*}
The $t\times t$ minors of $\bar{F}_I(x,a)$ still defines
$\widetilde{\Tj}_I(X)$. Notice that we can choose of special minors
$\Delta_{i,j}$, where $i\in\{1,\dots,m\}$ and $j\notin I$, where each
row and each column have a single non zero entry, which is $1$ except
for the $ii_j$ entry which is
$f_{ij}(x)-\sum_{j=1}^{t-1}a_{ji}f_{ii_j}(x)$. This implies that
$\widetilde{\Tj}_I(X)$ is defined by the $(n-t+1)m$ equations 
$f_{ij}(x)-\sum_{j=1}^{t-1}a_{ji}f_{ii_j}(x) =0$. Hence it is defined by the
$1\times 1$ minors of the matrix $m\times(n-t+1)$ matrix $F'_I(x,a)$
with columns:
\begin{align*}
F'_i=
\begin{pmatrix}
f_{1i}(x)-\sum_{j=1}^{t-1}a_{ji}f_{1i_j}(x)\\
\vdots\\
f_{mi}(x)-\sum_{j=1}^{t-1}a_{ji}f_{mi_j}(x)\\
\end{pmatrix}
\text{ if } i\notin I.
\end{align*}

This does still not imply that $\widetilde{\Tj}_I(X)$ is a
determinantal singularity, since the codimension might not be
right. Even if $\widetilde{\Tj}_I(X)$ is a determinantal variety, it is
often not irreducible, and hence $\widetilde{\Tj}_I(X)\neq \Tj_I(X)$,
as we will see in the next examples. 

\begin{ex}\label{ex1}
Let $X$ be the determinantal singularity of type $(2,3,2)$ defined by
the following matrix
\begin{align*}
F_1(x,y,z,w)=
\begin{pmatrix}
w^l & y & x\\
z & w &y^k\\
\end{pmatrix}.
\end{align*}
For $k,l>2$.
In this case $\widetilde{\Tj}_I(X)$ is a determinantal variety for all
$I$. $\widetilde{\Tj}_I(X)\neq \Tj_I(X)$ in the chart defined by
$I=\{2\}$, in the other charts both $\widetilde{\Tj}_I(X)$ and $\Tj_I(X)$
are smooth. Now lets look closer on the equations in the chart defined
by $I=\{2\}$.
\begin{align*}
F'_{\{2\}}(x,y,z,w,a_1,a_3)=
\begin{pmatrix}
w^l -a_1y & x-a_3y\\
z -a_1w &y^k -a_3w\\
\end{pmatrix}.
\end{align*}
Notice that the equations $x-a_3y=0$ and $z-a_1w=0$ just define $x$ and $z$ as
holomorphic functions of the other variables, and give embeddings of
a $\C^4$ into $\C^6$. Now if we multiply the equations $y^k-a_3w=0$
and $w^l-a_1y=0$ we get:
\begin{align*}
0 &= (y^k-a_3w)(w^l-a_1y)=y^kw^l-a_1y^{k+1}-a_3w^{l+1}+a_1a_3yw \\ &=
y^kw^l-a_1a_3yw-a_1a_3yw+a_1a_3yw=yw(y^{k-1}w^{w-1}-a_1a_3). 
\end{align*}
Hence we see that $\widetilde{\Tj}_I(X)$ is not irreducible. $y=0$ and
$w=0$ both define the fibre $(\tilde{\pi}^{Tj}_I)\inv(0)$ which is two
dimensional and therefore can not be a subset of $\Tj_I(X)$. Therefore,
$\Tj_I(X)$ is given by the equations $y^{k-1}w^{w-1}-a_1a_3=0$, $w^l
-a_1y =0$ and $y^k -a_3w=0$. Hence it is a determinantal singularity
of the same type as $X$ given by the matrix.
\begin{align*}
\begin{pmatrix}
w^{l-1} & y & a_3\\
a_1 & w &y^{k-1}\\
\end{pmatrix}.
\end{align*}
\end{ex}

\begin{ex}\label{ex2}
Let $X\subset \C^4$ be a determinantal singularity of type $(3,2,2)$ given by
\begin{align*}
F_2(x,y,z,w)=
\begin{pmatrix}
w^l & z\\
y & w \\
x &y^k\\
\end{pmatrix}.
\end{align*}
For $k,l>2$. The $\widetilde{\Tj}_I(X)$ is given in the two charts
$I=\{1\},\{2\}$ by the matrices
\begin{align*}
F'_{\{ 1\}}(x,y,z,w,a_1)=
\begin{pmatrix}
z -a_1w^l\\
w -a_1y\\
y^k-a_1x\\
\end{pmatrix}
\text{ and }
F'_{\{ 2\}}(x,y,z,w,a_2)=
\begin{pmatrix}
w^l -a_2z\\
y -a_2w\\
x -a_2y^k\\
\end{pmatrix}
\end{align*}
In this case we see that $\widetilde{\Tj}_I(X)= \Tj_I(X)$, and hence
the Tjurina transform of $X$ is a complete intersection.
\end{ex}

Notice that the singularities in Example \ref{ex1} and \ref{ex2} are
the same, it is just their representations as determinantal
singularities that are different. In fact the difference is that
$F_1(x,y,z,w)=F_2(x,y,z,w)^T$. 

Let us define $\Tj^T(X)$.
\begin{defn}
Let $X$ be a determinantal singularity of type $(m,n,t)$ given by
$\morf{F}{\C^N}{M_{m,n}}$, define $\morf{C}{X_{reg}}{\G(t-1,m)}$ as
the map that sends $x$ into the column space of $F(x)$. Then we define
$\Tj^T(X)$ of $X$ as
\begin{align*}
\Tj^T(X) &= \overline{\bigg{\{} \big(x,W\big)\in X_{reg}\times\G(t-1,n)\ \vert\
    W=C(x)\bigg{\}} }\subset X\times \G(t-1,m),
\end{align*}
and we define the map $\morf{\pi^{Tj^T}}{\Tj^T(X)}{X}$ as the projection
to the first factor.
\end{defn}
This definition gives us that $\Tj^T(X)=\Tj(X^T)$, where $X^T$ is $X$
but defined as a
determinantal singularity by $\morf{F(x)^T}{\C^N}{M_{n,m}}$. This
means that we can define $\widetilde{\Tj}_I^T(X)$ as for $\Tj(X)$,
either by setting  $\widetilde{\Tj}_I^T(X)= \widetilde{\Tj}_I(X^T)$ or
by defining it using $\bar{F}_I^T(x,a):=\left(
\begin{array}{@{} c | c @{}}
F(x) & A_I^T
\end{array}
\right)$,
 where $I$ now is a subset of $1,\dots,m$.

This of course immediately gives us the following results.
\begin{prop} $\widetilde{\Tj}_I^T(X)$ is a determinantal singularity if
  and only if $\dim X^s\leq N-(m-t+1)(n-s+1)$ for all $s\in 1,\dots,t$.
\end{prop}

\begin{prop}
$\widetilde{\Tj}_I^T(X)=\Tj^T_I(X)$ if and only if $\dim X^s<
N-(m-t+1)(n-s+1)$ for all $s\in 1,\dots,t-1$.
\end{prop}

Notice that this definition of $\Tj^T(\M{t})$ is the same as the one
we gave earlier, since the column space of a matrix is the same as its
image.

The next example shows that just like the blow-up and the Nash
transform, the Tjurina transform of a normal singularity need not be
normal and that the dimension of the singular set can increase under the
Tjurina transform.
\begin{ex}[$\Tj(X)$ need not be normal]
Let $X$ be the hypersurface singularity given by
$z^2-x^4-x^2y^3-x^2y^5-y^8=0$ it can be given as a determinantal
singularity of type $(2,2,2)$ by the matrix $\left(\begin{smallmatrix}
 z & x^2+y^3\\
 x^2+y^5 & z\\
\end{smallmatrix}\right)$. We get that the Tjurina transform is
\begin{align*}
F'_{\{ 1\}}(x,y,z,a_2)=
\begin{pmatrix}
x^2+y^3 -a_2z\\
z -a_2(x^2+y^5)\\
\end{pmatrix}
\text{ and }
F'_{\{ 2\}}(x,y,z,a_1)=
\begin{pmatrix}
z -a_1(x^2+y^3)\\
x^2+y^5 -a_1z\\
\end{pmatrix}.
\end{align*}  
In the first chart we can by change of coordinates, see that we have
the hypersurface singularity $x^2+y^3-a_2^2(x^2+y^5)=0$, which has all
of the $a_2$-axis as its singular set. In the same way the second
chart gives us the hypersurface $x^2+y^5-a_1^2(x^2+y^3)=0$, which has
the $a_1$-axis as its singular set. Hence $\Tj(x)$ have singularities
of codimension 1, and is, therefore, not normal. It also illustrates
that the singular set of $\Tj(X)$ might have larger dimension than the
singular set of $X$.
\end{ex}

We saw in Section \ref{modelsing} that for the model determinantal
singularities $\Na(\M{t}) \cong
\Tj(\M{t})\times_{(\M{t}-\M{t-1})}\Tj^T(\M{t})$. Is this then true in general?
Is $\Na(X) \cong \Tj(X)\times_{X_{reg}}\Tj^T(X)$? The answer is
unfortunately no as we can see in the following. Let $X$ be the
determinantal singularity defined in Example \ref{ex1}. There we saw
that the exceptional divisor of $\Tj(X)$ consist of two irreducible
components. In Example \ref{ex2} we got that the exceptional divisor
of $\Tj^T(X)$ is a single irreducible curve. Hence the exceptional divisor of
$\Tj(X)\times_{X_{reg}}\Tj^T(X)$ consists of three irreducible
curves. But in \cite{tjurina} Tjurina shows that $X$ is a minimal
surface singularity with the following dual resolution graph.
$$
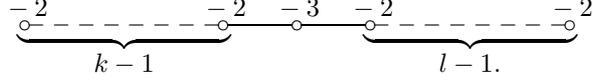
\xymatrix@R=6pt@C=24pt@M=0pt@W=0pt@H=0pt{
\overtag{\Circ}{-2}{8pt}\dashto[rr] &
{\hbox to 0pt{\hss$\underbrace{\hbox to 80pt{}}$\hss}}&
\overtag{\Circ}{-2}{8pt}\lineto[r] &
\overtag{\Circ}{-3}{8pt}\lineto[r] &
\overtag{\Circ}{-2}{8pt}\dashto[rr] &
{\hbox to 0pt{\hss$\underbrace{\hbox to 80pt{}}$\hss}}&
\overtag{\Circ}{-2}{8pt}\\
&{k-1}&&&&{l-1}.}$$
Following the work of Spivakovsky \cite{spivakovsky} the irreducible
components of the exceptional divisor of the normalized Nash transform of a
surface singularity corresponds to the irreducible components of the
exceptional divisor intersecting the strict transform of the polar curve of
a generic plane projection. By Theorem 5.4 in Chapter III of
\cite{spivakovsky} we find that the polar of a generic plane
projection of $X$ intersects the exceptional divisor in two different
components. This implies that the exceptional
divisor of $\Na(X)$ has at most two components, since the number of
components can not decrease under normalization. Hence $\Na(X)$ and
$\Tj(X)\times_{X_{reg}}\Tj^T(X)$ have non isomorphic exceptional divisors,
and can, therefore, not be isomorphic as transformations.

\section{When is the Tjurina transform a Complete Intersection}\label{completeintersection}

In  Lemma 3.14 of their articel \cite{fruhbiskrugerzach}
Fr\"uhbis-Kr\"uger and Zach find conditions under which the Tjurina
transform of Cohen-Macauley codimension 2 singularities in $\C^5$ only
has isolated singularities. They also notice in Remark 3.16 that in all the
cases of simple isolated Cohen-Macauley codimension 2 singularities
they consider, the Tjurina transforms are isolated local complete
intersection. In this section we  consider the second question and
give some general condition on when the Tjurina transform of an
EIDS is a local complete intersection.

If $X$ is an EIDS, remember that means that $F$ is transverse to all
strata of $\M{t}$ in a punctured neighbourhood of the origin, then we
can get the following result concerning the Tjurina transform.

\begin{prop}\label{dependingonN}
Let $X\subset\C^N$ be an EIDS of type $(m,n,t)$, then $\Tj(X)$ is a
local complete intersection if $N-m(n-t+1)>\dim X^1$ and $\Tj^T(X)$ is
a local complete intersection if $N-n(m-t+1)>\dim X^1$.
\end{prop}
\begin{proof}
To show that $\Tj(X)$ is a local complete
intersection, it is enough to show that $\Tj_I(X)$ is a complete
intersection for all $I$. To do this we show that
$\Tj_I(X)=\widetilde{\Tj}_I(X)$. First notice that being an EIDS implies that
$\dim X^s =N-(m-s+1)(n-s+1)<N-(m-s+1)(n-t+1)$ for all $s\in 2,\dots
t-1$. So for $X$ to satisfy Proposition \ref{tjtilde=tj} we just need
that $\dim X^1< N-m(n-t+1)$ which follows from the assumption. So
$\Tj_I(X)=\widetilde{\Tj}_I(X)$  and $\widetilde{\Tj}_I(X)$ is a
complete intersection. Hence $\Tj(X)$ is a local complete
intersection.

The proof for $\Tj^T(X)$ is similar, just exchange $n$ and $m$.
\end{proof}

The assumption on $N$ can replaced by assumption on $t$ and the strata
of $X$ as seen in the next theorem.

\begin{thm}\label{tgeaterthantree}
Let $X$ be an EIDS of type $(m,n,t)$, where $t\geq 3$ and $X^2\neq
\emptyset$. Then at least one of $\Tj(X)$ and $\Tj^T(X)$ is a local
complete intersection. 
\end{thm}

\begin{proof}
First notice that since $t\geq 3$ one of the following two inequalities holds
$n-1<m(t-2)$ or $m-1<n(t-2)$. We will first show that if the first
equation holds, then $\Tj(X)$ is a local complete intersection.

Assume that $n-1<m(t-2)$. To show that that $\Tj(X)$ is a complete
intersection, we just need to show that $\dim X^1<N-m(n-t+1)$ by Proposition
\ref{dependingonN}. Now $\dim X^1<\dim X^2= N
-(m-1)(n-1) = N-mn+m+n-1< N-mn+m+m(t-2)= N-m(n-t+1)$. So $\Tj(X)$
is a local complete intersection.

If $m-1<n(t-2)$ then the same argument with exchanging $m$ and $n$,
shows that $\Tj^T(X)$ is a local complete intersection.
\end{proof}

As we saw in Example \ref{ex1} and Example \ref{ex2} the theorem can still
hold if $t<3$, but next we will give an example with $t=2$ where we
have that $\Tj_I(X)\neq \widetilde{\Tj}_I(X)$ and $\Tj_J^T(X)\neq
\widetilde{\Tj}^T_J(X)$ for all $I,J$. But in the example both
$\Tj(X)$ and $\Tj^T(X)$ are complete intersections.

\begin{ex}\label{ex3}
Let $X\subset \C^3$ be a determinantal singularity of type $(3,2,2)$ given by
\begin{align*}
F_3(x,y,z,w)=
\begin{pmatrix}
z & y & x^{k-3}\\
0 & x & y \\
\end{pmatrix}.
\end{align*}
For $k>4$. The $\widetilde{\Tj}_I(X)$ is given in the three charts
$I=\{1\},\{2\},\{3\}$. In the first chart the matrix is
\begin{align*}
F'_{\{ 1\}}(x,y,z,a_2,a_3)=
\begin{pmatrix}
y -a_2z & x^{k-3}-a_3z\\
x & y-a_3x\\
\end{pmatrix}.
\end{align*}
We see that $\widetilde{\Tj}_{\{ 1\}}(X)$ is the fibre over $0$ (given
by $x=y=z=0$) union
the $z$-axis (given by $x=y=a_2=a_3=0$), so we get that $\Tj_{\{
  1\}}(X)$ is the $z$-axis. 

In the second chart we get
\begin{align*}
F'_{\{ 2\}}(x,y,z,w,a_1,a_3)=
\begin{pmatrix}
z -a_2y & x^{k-3}-a_3y\\
-a_1x & y -a_3x\\
\end{pmatrix}.
\end{align*}
Here we see that $\widetilde{\Tj}_{\{ 2\}}(X)$ is the fibre $0$ (given
by $x=y=z=0$) union the curve singularity given by $x^{k-4}-a_3^2=0$,
$y=a_3x$ and $a_1=z=0$. Hence $\Tj_{\{ 2\}}(X)$ is a $A_{k-5}$ plane
curve singularity embedded in $\C^5$.

In the last chart we
\begin{align*}
F'_{\{ 3\}}(x,y,z,w,a_1,a_2)=
\begin{pmatrix}
z -a_1x^{k-3} & y-a_2x^{k-3}\\
-a_1y & x -a_2y\\
\end{pmatrix}.
\end{align*}
Now we see that $\widetilde{\Tj}_{\{ 2\}}(X)$ is the fibre $0$ (given
by $x=y=z=0$) union the curve given by $1-a_2^2x^{k-4}=0$,
$y=a_2x^{k-3}$ and $a_1=z=0$. Hence $\Tj_{\{ 2\}}(X)$ is a smooth
curve in this chart.

So $\Tj(X)$ is a line disjoint union a $A_{k-5}$ curve, and the fibre over $0$
is $2$ dimensional.

If we calculate $\widetilde{\Tj}^T_I(X)$ in the charts $\{ 1\}$ and $\{
2\}$. We get
\begin{align*}
F'_{\{ 1\}}(x,y,z,a_2)=
\begin{pmatrix}
-a_2z\\
x -a_2y\\
y-a_2x^{k-3}\\
\end{pmatrix}
\text{ and }
F'_{\{ 2\}}(x,y,z,w,a_1)=
\begin{pmatrix}
z\\
y -a_1x\\
x^{k-3} -a_1y\\
\end{pmatrix}.
\end{align*}
We see that in the first chart we have a line union the fibre over
$0$ and in the second chart we have an $A_{k-5}$ curve singularity
union the fibre over zero.

So in this case we have that $\Tj(X)$ and $\Tj^T(X)$ are the same, a
line disjoint union an $A_{k-5}$. Notice that in this case $\Tj(X)$ is
also local complete intersection. If we consider that $X$ is a line through
the singular point of a $A_{k-4}$, we see that the transformation
separates the line from the singularity, but makes the singularity
somewhat worse.
\end{ex}

In Theorem \ref{tgeaterthantree} we saw that if $t\geq 3$ then one of
$\Tj(X)$ or $\Tj^T(X)$ is a local complete intersection, the case
$t=1$ is not interesting, because in this case $\Tj(X)=\Tj^T(X)=X$ and
$X$ is a complete intersection. The next proposition will explain the
case for $t=2$.

\begin{prop}
Let $X$ be an EIDS of type $(m,n,2)$, then one of $\Tj(X)$ or
$\Tj^T(X)$ is a local complete intersection if $\min(n,m)\leq \dim X -
\dim X^1$.
\end{prop}

\begin{proof}
To prove that $X$ is a complete intersection we just need to see that
$\dim X^1<N-m(n-t+1)=N-m(n-1)$ by Proposition \ref{dependingonN}. But
$(m-t+1)(n-t+1)=(m-1)(n-1)=\codim X$, hence
the inequality becomes $\dim X^1< (m-1)(n-1)+dim X -m(n-1)$. Hence
$\Tj(X)=\widetilde{\Tj}(X)$ and hence a complete intersection if
$n-1<\dim X-\dim X^1$. The case $\Tj^T(X)= \widetilde{\Tj^T}(X)$ is
gotten by exchanging $n$ and $m$ and the result follows.
\end{proof}

\begin{cor}
Let $X$ be an EIDS of type $(m,n,2)$ with an isolated singularity,
then one of $\Tj(X)$ or $\Tj^T(X)$ is a local complete intersection if
$\min(n,m)\leq \dim X$. 
\end{cor}

These result is only in the one direction, because what we really
prove is that if the inequalities are satisfied, then
$\Tj(X)=\widetilde{\Tj}(X)$ or $\Tj^T(X)= \widetilde{\Tj^T}(X)$. But
$\Tj(X)$ or $\Tj^T(X)$ could still be local complete intersections,
even if this is not true. 

\section{Using Tjurina Transform to Resolve Hypersurface
  Singularities}\label{hypersurface}

In the previous section we saw that very often the Tjurina transform is a
complete intersection, which means that one can not get a resolution
by using only the Tjurina transform. Notice also that in
several of the examples $\Tj(X)$ is normal, so using only Tjurina
transform and normalizations will also not produce a resolution. In the
next example we will look at the case of the $A_n$ surface
singularities and see that it might not be completely impossible to use
the Tjurina transform to achieve a resolution. 

\begin{ex}[$A_n$ singularities]\label{ansing}
We will in this example see how different representations of the simple
$A_n$ singularity can lead to different Tjurina transforms. 

First we can of course represent $A_n$ as a determinantal singularity
of type $(1,1,1)$, then the Tjurina transform of $A_n$ is just $A_n$
itself, by Proposition \ref{completeintersection}. But we can also
represent $A_n$ as a determinantal singularity of type $(2,2,2)$ like
the following:
\begin{align*}
F(x,y,z)=
\begin{pmatrix}
 x & z^l\\
 z^{n-l+1} & y \\
\end{pmatrix},
\end{align*}
where $0<l\leq n$. In this case we get that the Tjurina transform is given by:
\begin{align*}
F'_{\{ 1\}}(x,y,z,a_2)=
\begin{pmatrix}
z^l -a_2x\\
y -a_2z^{n-l+1}\\
\end{pmatrix}
\text{ and }
F'_{\{ 2\}}(x,y,z,a_1)=
\begin{pmatrix}
x -a_1z^l\\
z^{n-l+1} -a_1y\\
\end{pmatrix}.
\end{align*}
So we see that $\Tj(A_n)$ using this representations have an $A_{l-1}$
and an $A_{n-l}$ singularity, so we have simplified the
singularity. It is clear, that by writing these new $A_m$ singularities
as determinantal singularities of type $(2,2,2)$, we can apply the
Tjurina transform again to simplify the singularity. By repeatedly
doing this we can resolve the $A_n$ singularity. 
\end{ex}

As we can see in Example \ref{ansing} the Tjurina transform depends
not only on the singularity type of $X$ but we also get different
transforms if we have different matrix presentations of the same type.

We will in the next example show how to obtain a resolution trough
repeated Tjurina transform changing the determinantal type and
matrix presentation. By this we mean that the Tjurina transform gives
us a complete intersection of the form $(m,n,1)$, which by change of
coordinates locally can be seen as a hypersurface. We will then write
this hypersurface as a determinantal singularity of type $(t,t,t)$.

\begin{ex}[$E_7$ singularity]
The simple surface singularity $E_7$ can be defined by the following
equation $y^2+ x(x+z^3)=0$. This can be seen as a determinantal
singularity of type $(2,2,2)$ given by the following matrix:
$\left(\begin{smallmatrix}
 y & x^2+ z^3\\
 -x & y \\
\end{smallmatrix}\right)$. We then perform the Tjurina transform and get:
\begin{align*}
F'_{\{ 1\}}(x,y,z,a_2)=
\begin{pmatrix}
x^2+z^3 -a_2x\\
y -a_2x\\
\end{pmatrix}
\text{ and }
F'_{\{ 2\}}(x,y,z,a_1)=
\begin{pmatrix}
y -a_1(x^2+z^3)\\
-x -a_1y\\
\end{pmatrix}.
\end{align*}
By changing coordinates we see that $F'_{\{ 1\}}$ is equivalent to the
hypersurface $x^2+z^3+w^2x=0$, which has a singular point at $(0,0,0)$,
and $F'_{\{ 2\}}$ is equivalent to the hypersurface $x+v^2(x^2+z^3)=0$
which is non singular. 

So we will continue working in the first chart, and we will denote this
singularity $\Tj(E_7)$. In the coordinates $x^2+z^3+w^2x=0$ the
exceptional divisor $E_1=(\pi^{Tj})\inv(0)$ is given by $x=z=0$. We
now write $\Tj(E_7)$ as the matrix $\left(\begin{smallmatrix}
 x & z^2\\
 -z & x+w^2 \\
\end{smallmatrix}\right)$ and perform the Tjurina transform.
\begin{align*}
F'_{\{ 1\}}(x,z,w,a_2)=
\begin{pmatrix}
z^2 -a_2x\\
x+w^2 +a_2z\\
\end{pmatrix}
\text{ and }
F'_{\{ 2\}}(x,z,w,a_1)=
\begin{pmatrix}
x -a_1z^2\\
-z -a_1(x+w^2)\\
\end{pmatrix}.
\end{align*}
The first chart is equivalent to the hypersurface $z^2+yw^2+y^2z=0$
which has a singularity at $(0,0,0)$, and the second chart is
equivalent to $z+v(vz^2+w^2)=0$ which is smooth. The exceptional
divisor consist of two components the strict transform
of exceptional divisor from before (which we by abuse of notation
denote by $E_1$) is given by $z=y=0$ and the new addition $E_2$ given by
$x=w=0$. They intersect each other in the singular point.

We will continue in the first chart and denote this singularity by
$\Tj^2(E_7)$. It can be given by the matrix $\left(\begin{smallmatrix}
 y & z\\
 -z & yz+w^2 \\
\end{smallmatrix}\right)$ as a determinantal singularity of type
$(2,2,2)$. The Tjurina transform is
\begin{align*}
F'_{\{ 1\}}(y,z,w,a_2)=
\begin{pmatrix}
z -a_2y\\
yz+w^2 +a_2z\\
\end{pmatrix}
\text{ and }
F'_{\{ 2\}}(y,z,w,a_1)=
\begin{pmatrix}
y -a_1z\\
-z -a_1(yz+w^2)\\
\end{pmatrix}.
\end{align*}
I the first chat we have the hypersurface $xy^2+w^2+x^2y=0$ which has
$(0,0,0)$ as its only singular point. The second chart is
$z+v(vz^2+w^2)=0$ which is smooth. The exceptional divisor consist of
$E_1$ given by $z=v=0$ (so it only exists in the second chart), $E_2$
given by $x=w=0$ and the new $E_3$ given by $y=w=0$. $E_1$ and $E_2$
does not meet, but $E_3$ intersects them both, $E_1$ in a smooth
point and $E_2$ in the singular point. 

We represent the singularity $\Tj^3(E_7)$ as the matrix $\left(\begin{smallmatrix}
 xy & w\\
 -w & x+y \\
\end{smallmatrix}\right)$. The Tjurina transform is then
\begin{align*}
F'_{\{ 1\}}(x,y,w,a_2)=
\begin{pmatrix}
w -a_2xy\\
x+y +a_2w\\
\end{pmatrix}
\text{ and }
F'_{\{ 2\}}(x,y,w,a_1)=
\begin{pmatrix}
xy -a_1w\\
-w -a_1(x+y)\\
\end{pmatrix}.
\end{align*}
In the first chart we have the hypersurface $x+y+v^2xy=0$ which is
smooth. The second chart gives the hypersurface singularity
$xy+z^2(x+y)=0$, which has a singular point at $(0,0,0)$. $E_1$ does
not exists in $\Tj^3(E_7)$, but intersects $E_3$ in a smooth point in
the other charts. $E_2$ is given by $x=z=0$, $E_3$ is given by $y=z=0$
and the new $E_4$ is given by $x=y=0$. $E_2$, $E_3$ and $E_4$
intersect each other in the singular point. 

Next we can represent the singularity $\Tj^4(E_7)$ by the matrix $\left(\begin{smallmatrix}
 x & z(x+y)\\
 -z & y \\
\end{smallmatrix}\right)$. The Tjurina transform is then
\begin{align*}
F'_{\{ 1\}}(x,y,z,a_2)=
\begin{pmatrix}
z(x+y) -a_2x\\
y +a_2z\\
\end{pmatrix}
\text{ and }
F'_{\{ 2\}}(x,y,z,a_1)=
\begin{pmatrix}
x -a_1z(x+y)\\
-z -a_1y\\
\end{pmatrix}.
\end{align*}  
The first chart gives the hypersurface $zx-wx-wz^2=0$ which has a
singular point at $(0,0,0)$, and the second chart gives $x-v^2y(x+y)=0$
which is smooth. The exceptional divisor consist of $E_2$ given by
$z=v=0$ so not in the chart that contains the singularity, $E_3$ given
by $z=w=0$, $E_4$ given by $x=w=0$ and $E_5$ given by $X=z=0$. $E_2$
intersects $E_5$ in a smooth point, $E_3$, $E_4$ and $E_5$ intersect
each other in the singular point, and $E_3$ intersects $E_1$ in a
smooth point outside these charts.

We can represent $\Tj^5(E_7)$ by the matrix  $\left(\begin{smallmatrix}
 z & x\\
 w & x-wz \\
\end{smallmatrix}\right)$. In this case the Tjurina transform is
\begin{align*}
F'_{\{ 1\}}(x,z,w,a_2)=
\begin{pmatrix}
x -a_2z\\
x-wz -a_2w\\
\end{pmatrix}
\text{ and }
F'_{\{ 2\}}(x,z,w,a_1)=
\begin{pmatrix}
z -a_1x\\
w -a_1(x-wz)\\
\end{pmatrix}.
\end{align*}  
The first chart gives the hypersurface $yz-wz+yw=0$ which has a
singularity at $(0,0,0)$, and the second chart gives the smooth
hypersurface $w-vx-v^2wx=0$. The exceptional divisor consists of $E_1$ and
$E_2$ that do not appear in any of these charts, $E_3$ given by
$z=v=0$ (so only appearing in the second chart), $E_4$ given by
$w=y=0$, $E_5$ given by $z=y=0$ and $E_6$ given by $w=z=0$. $E_3$
intersects $E_1$ and $E_6$ in different smooth points, $E_2$
intersects $E_5$ in a smooth point, $E_4$, $E_5$ and $E_6$ intersect
each other in the singular point. 

For $\Tj^6(E_7)$ we use the matrix $\left(\begin{smallmatrix}
 y & w\\
 z & z+w \\
\end{smallmatrix}\right)$. We get that the Tjurina transform is
\begin{align*}
F'_{\{ 1\}}(y,z,w,a_2)=
\begin{pmatrix}
w -a_2y\\
z+w -a_2z\\
\end{pmatrix}
\text{ and }
F'_{\{ 2\}}(y,z,w,a_1)=
\begin{pmatrix}
y -a_1w\\
z -a_1(z+w)\\
\end{pmatrix}.
\end{align*}  

The first chart is the smooth hypersurface $z+xy-xz=0$, and the second
chart is $z-vz-y=0$ which is also smooth. So we have reach a resolution
of $E_7$. The exceptional divisors consist of $E_1,\dots,E_7$, where
only $E_4\dots, E_7$ appears in the last two charts. $E_4$ is given by
$y=x-1=0$, $E_5$ is given by $z=v=0$, $E_6$ is given by $z=x=0$ and
$E_7$ is given by $z=y=0$. $E_7$ intersects $E_4$, $E_5$ and $E_6$ in
three different smooth points, $E_2$ intersects $E_5$ in a smooth
point, and $E_3$ intersects $E_1$ and $E_6$ in two different smooth
points. If we represent the exceptional divisor by a dual resolution
graph (where vertices represent the curves and edges represents the
intersection points) we get:
$$
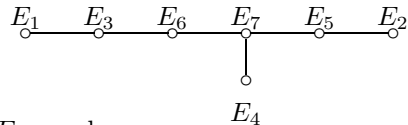
\xymatrix@R=6pt@C=24pt@M=0pt@W=0pt@H=0pt{
\\
\overtag{\Circ}{E_1}{8pt}\lineto[r] &
\overtag{\Circ}{E_3}{8pt}\lineto[r] &
\overtag{\Circ}{E_6}{8pt}\lineto[r] &
\overtag{\Circ}{E_7}{8pt}\lineto[r]\lineto[dd] &
\overtag{\Circ}{E_5}{8pt}\lineto[r] & 
\overtag{\Circ}{E_2}{8pt} & \\
&&&&\\
&&&\undertag{\Circ}{E_4}{2pt}\\
&&&&\hbox to 0pt {\hss}}
$$
which is indeed the $E_7$ graph.
\end{ex}

One can also use this method to produce resolutions of the $D_n$ and $E_6$
singularities, and probably many more. But it is not always possible
to use this method, the $E_8$ given by $x^2+y^3+z^5$ can not be written
as a the determinant of a $2\times 2$ matrix which is $0$ at the origin of
$\C^3$, nor can it be written as the determinant of a larger matrix
such the value at the origin is $0$. If the value at the origin is not zero,
then the Tjurina transform does not improve the singularity, it only
changes variables. 

\bibliography{general}

\end{document}